\documentclass[12pt]{amsart}
\usepackage{bbm}

\DeclareMathOperator{\cchar}{char} %
\DeclareMathOperator{\GL}{GL} %
\DeclareMathOperator{\Lie}{Lie} %

\newcommand\Z{{\mathbb Z}}
\newcommand\KK{{\mathbb K}}

\renewcommand\l{{\mathfrak l}}
\newcommand\g{{\mathfrak g}}
\renewcommand\b{{\mathfrak b}}
\newcommand\h{{\mathfrak h}}
\renewcommand\c{{\mathfrak c}}
\newcommand\p{{\mathfrak p}}
\newcommand\q{{\mathfrak q}}
\renewcommand\u{{\mathfrak u}}
\newcommand\z{{\mathfrak z}}

\newcommand\CC{\mathcal C}
\newcommand\CN{{\mathcal N}}
\newcommand\CS{{\mathcal S}}
\newcommand\bk{\mathbbm k}

\renewcommand\mod{{\mathrm{mod}}}

\newcommand{\twobytwo}[4]
 {{\left(\begin{array}{ll} #1 & #2 \\ #3 & #4 \end{array}\right)}}

\usepackage{fullpage}

\numberwithin{equation}{section}

\theoremstyle{plain}
\newtheorem{lem}[equation]{Lemma}
\newtheorem{thm}[equation]{Theorem}
\newtheorem{cor}[equation]{Corollary}

\theoremstyle{definition}
\newtheorem{strategy}[equation]{Strategy}

\theoremstyle{remark}
\newtheorem{rem}[equation]{Remark}

\title
{On commuting varieties of nilradicals of Borel subalgebras of
reductive Lie algebras}

\author[S.\ M.\ Goodwin] {Simon M.~Goodwin}
\address{School of Mathematics,
University of Birmingham,
Birmingham, B15 2TT,
United Kingdom}
\email{s.m.goodwin@bham.ac.uk}

\author[G.\ R\"ohrle]{Gerhard R\"ohrle}
\address{Fakult\"at f\"ur Mathematik, Ruhr-Universit\"at Bochum, D-44780
Bochum, Germany} \email{gerhard.roehrle@rub.de}


\begin{document}

\begin{abstract}
Let $G$ be a connected reductive algebraic group defined over
an algebraically closed field $\bk$ of characteristic zero.
We consider the commuting variety $\CC(\u)$ of the
nilradical $\u$ of the Lie algebra $\b$ of a Borel subgroup $B$ of $G$.
In case $B$ acts on $\u$ with only a finite
number of orbits, we verify that $\CC(\u)$ is equidimensional
and that the irreducible components are in correspondence with the {\em distinguished} $B$-orbits in $\u$.
We observe that in general $\CC(\u)$ is not equidimensional, and
determine the irreducible components of $\CC(\u)$ in
the minimal cases where there are infinitely many $B$-orbits
in $\u$.
\end{abstract}

\maketitle

\section{Introduction}
\label{sect:intro}

Let $G$ be a connected reductive algebraic group defined over
an algebraically closed field $\bk$ of characteristic zero
and let $\g = \Lie G$ be its Lie algebra.
It was proved by Richardson that the commuting variety
$$
\CC(\g) = \{(x,y) \in \g \times \g \mid [x,y] = 0\}
$$
of $\g$ is irreducible, see \cite{rich2}.
This fact was generalized to positive good characteristic
by Levy in \cite{levy}.  In \cite{premet},
Premet showed that the commuting variety
$\CC(\CN) = \CC(\g) \cap (\CN \times \CN)$
of the nilpotent cone $\CN$ of $\g$ is equidimensional,
where the  irreducible components are in correspondence with the distinguished nilpotent $G$-orbits in $\CN$;
this theorem was proved also in positive good characteristic.

In this short note we consider the commuting variety of the Lie algebra
of the unipotent radical of a Borel subgroup of $G$.  To explain this
further we introduce some notation.
Let $B$ be a Borel subgroup of $G$ with unipotent radical $U$ and write $\b$ and $\u$ for
the Lie algebras of $B$ and $U$, respectively.  The {\em commuting variety of $\u$} is
$$
\CC(\u) = \{(x,y) \in \u \times \u \mid [x,y]=0 \}.
$$
For $e \in \u$, we write $\c_\b(e)$ and $\c_\u(e)$ for the centralizer of $e$ in $\b$ and $\u$, respectively.
We define
$$
\CC(e) = \overline{B \cdot (e,\c_\u(e))} \subseteq \CC(\u)
$$
to be the Zariski closure of the $B$-saturation of $(e,\c_\u(e))$ in $\CC(\u)$; it is easy
to see that $\CC(e)$ is irreducible and $\dim \CC(e) = \dim B - \dim \c_\b(e) + \dim \c_\u(e)$.
We say that $e \in \u$ is \emph{distinguished} provided
$\c_\b(e) = \c_\u(e)$, and
note that for $e$ distinguished we have  $\dim \CC(e) = \dim B$.

Below is an analogue of Premet's theorem from \cite{premet}
for the case when $B$ acts
on $\u$ with a finite number of orbits.

\begin{thm}
\label{thmmain}
Suppose that $B$ acts on $\u$ with a finite number of orbits.
Let $e_1,\dots,e_r$ be representatives of the distinguished $B$-orbits in $\u$.
Then
$$
\CC(\u) = \CC(e_1) \cup \dots \cup \CC(e_r)
$$
is the decomposition of the commuting variety $\CC(\u)$ into its irreducible components.
In particular, $\CC(\u)$ is equidimensional
of dimension $\dim B$.
\end{thm}

The cases where $B$ acts on $\u$ with a finite number of orbits are known,
thanks to work by B\"urgstein and Hesselink \cite{burgsteinhesselink}
and Kashin \cite{kashin}.  This is the case precisely when the length
$\ell(\u)$ of the descending central series of $\u$ is at most 4.
Thus if $\g$ is simple, this is the case precisely
when $\g$ is of type $A_1$, $A_2$, $A_3$, $A_4$ or $B_2$.

We also consider the cases where
$\ell(\u) = 5$, so for $\g$ simple,
$\g$ is of type $A_5$, $B_3$, $C_3$, $D_4$ or $G_2$.
In these minimal cases where there are infinitely many $B$-orbits
in $\u$, we describe the irreducible components of $\CC(\u)$
in Section \ref{s:cu2}.
We note that in these cases,
$\CC(\u)$ is no longer equidimensional.
In fact, we observe that
$\CC(\u)$ is \emph{never} equidimensional when there are infinitely
many $B$-orbits in $\u$,
see Lemma \ref{L:notequi}.
This demonstrates that the situation
is considerably more subtle in the infinite orbit case
and there appears to be no obvious parametrization
of the irreducible components.

Our methods are also applicable to the case where $\u$ is the
Lie algebra of the unipotent radical of a parabolic subgroup $P$ of $G$.
There are examples of such situations where $P$ acts with finitely many
orbits on $\u$ yet $\CC(\u)$ is not equidimensional,
see Remark \ref{rem:parabolic}.

For simplicity, we assume that
$\cchar \bk = 0$ (or at least that $\cchar \bk$ is sufficiently large), though
with additional work, it is strongly expected that the results remain true in good characteristic.

\section{Generalities about commuting varieties}
\label{s:prelims}

For this section, we work in the following setting.
Let $P$ be a connected algebraic
group over $\bk$ and $U$
a normal subgroup of $P$;
we write $\p$ and $\u$ for the Lie algebras of $P$ and $U$, respectively.
The group $P$ acts on $\p$ and $\u$ via the adjoint action.
For $x \in \p$ and any subgroup $H$ of $P$,
we denote the $H$-orbit of $x$ in $\p$ by $H \cdot x$,
the centralizer of $x$ in $H$ by $C_H(x)$ and
and the centralizer of $x$ in $\h = \Lie H$ by $\c_\h(x)$.

Let $P$ act  diagonally on $\u \times \u$.
The {\em commuting variety of $\u$} is
the closed, $P$-stable subvariety of $\u \times \u$, given by
$$
\CC(\u) = \{(x,y) \in \u \times \u \mid [x,y]=0 \}.
$$

We recall that the
\emph{modality} of $U$ on $\u$ is defined to be
$$
\mod(U;\u) = \max_{i \in \Z_{\ge 0}}(\dim \u_i - i),
$$
where $\u_i = \{x \in \u \mid \dim U \cdot x = i\}$. 

Our first lemma gives
an expression for the dimension of $\CC(\u)$.

\begin{lem}
$\dim \CC(\u) = \dim U + \mod(U;\u)$.
\end{lem}

\begin{proof}
Consider $\CC(\u)_i = \CC(\u) \cap (\u_i \times \u)$.
Clearly, we have
$\dim \CC(\u)_i = \dim \u_i + (\dim U - i) = \dim U + (\dim \u_i - i)$
and $\CC(\u) = \bigcup_{i \in \Z_{\ge 0}} \CC(\u)_i$.
Therefore,
$$
\dim \CC(\u) = \max_{i \in \Z_\ge 0}\{\dim U + (\dim \u_i - i)\} = \dim U + \mod(U;\u).
$$
\end{proof}

For $e \in \u$, we define
$$
\CC(e) = \overline{P \cdot (e,\c_\u(e))} \subseteq \CC(\u)
$$
to be the Zariski closure of the $P$-saturation of $(e,\c_\u(e))$
in $\CC(\u)$. It is easy to see that
$\CC(e)$ is a closed irreducible $P$-stable subvariety of $\CC(\u)$
of dimension
\begin{equation}
\label{e:dimcx}
\dim \CC(e) = \dim P\cdot e + \dim \c_\u(e) = \dim P - (\dim \c_\p(e) - \dim \c_\u(e)).
\end{equation}

We define an action of  $\GL_2(\bk)$ on $\u \times \u$ by
$$
\twobytwo{\alpha}{\beta}{\gamma}{\delta} \cdot (x,y) = (\alpha  x + \beta y, \gamma x + \delta y),
$$
cf.\  part (1) of the proof of \cite[Prop.\ 2.1]{premet}.
Since any linear combination of two commuting elements from $\u$ gives in this way again a pair of commuting elements from $\u$,
it follows that $\GL_2(\bk)$ acts on $\CC(\u)$ and further, since $\GL_2(\bk)$ is connected it must stabilize each irreducible component of $\CC(\u)$.
This proves the following lemma.

\begin{lem}
\label{lem2}
The action of $\GL_2(\bk)$ on $\CC(\u)$ preserves each irreducible component.
In particular, each irreducible component is invariant under the involution
$\sigma : (x,y) \mapsto (y,x)$.
\end{lem}

For the remainder of this section apart from Remark~\ref{R:inf2}, we assume that there are finitely many $P$-orbits in $\u$, and we choose representatives $e_1,\dots,e_s$ of these orbits.
Then we have
\[
\CC(\u) = \CC(e_1) \cup \ldots \cup \CC(e_s).
\]
In particular, each irreducible component of $\CC(\u)$ is
of the form $\CC(e_i)$ for some $i$.

We proceed with some elementary lemmas.
We recall that under our assumption that $P$ acts on $\u$ with
finitely many orbits, there is a unique dense open $P$-orbit in $\u$.

\begin{lem}
\label{lem1}
$ $
\begin{itemize}
\item[(i)]
Let $e, e' \in \u$. If $\CC(e) \subseteq \CC(e')$, then $P \cdot e \subseteq \overline{P \cdot e'}$.
\item[(ii)]
If $e \in \u$ is in the dense open $P$-orbit,
then $\CC(e)$ is
an irreducible component of $\CC(\u)$.
\end{itemize}
\end{lem}

\begin{proof}
Let $\pi_1 : \u \times \u \to \u$ be the projection onto the first factor.
Since $\CC(e) \subseteq \CC(e')$, we have
$$
\overline{P \cdot e} = \pi_1(\CC(e)) \subseteq \pi_1(\CC(e')) = \overline{P \cdot e'},
$$
so (i) holds.
Part (ii) follows from (i).
\end{proof}

The next lemma is used to show
that certain $\CC(e)$ are not irreducible components of $\CC(\u)$.

\begin{lem}
\label{lem3}
Let $e \in \u$.
If $\CC(e)$ is an irreducible component of $\CC(\u)$, then
$\c_\u(e) \subseteq \overline{P \cdot e}$.
\end{lem}

\begin{proof}
The argument of part (2) in the  proof of \cite[Prop.\ 2.1]{premet} also applies in our case; we repeat
it here for the convenience of the reader. The projection $\pi_1 : \u \times \u \to \u$ on to the first factor maps an irreducible component $\CC(e)$ to $\overline{P\cdot e}$.
Consequently, by Lemma \ref{lem2}, we have
\[
\c_\u(e) \subseteq (\pi_1 \circ \sigma) \CC(e) = \overline{P\cdot e}.
\]
\end{proof}

We define
$$
d = \min_{e \in \u} \{\dim \c_\p(e) - \dim \c_\u(e)\},
$$
so we have $\dim \CC(\u) = \dim P - d$, by \eqref{e:dimcx}.
We say that $e \in \u$ is {\em distinguished} for $P$ if
$\dim \c_\p(e) - \dim \c_\u(e) = d$.  We assume that
our representatives of the $P$-orbits in $\u$
are chosen so that $e_1,\dots,e_r$ are the
representatives of the distinguished orbits.
The following lemma is immediate; we record it for ease of reference.

\begin{lem}
\label{lem8}
The irreducible components of $\CC(\u)$ of maximal dimension are $\CC(e_1),\dots,\CC(e_r)$.
\end{lem}

Assume from now on that there is a complementary subalgebra
$\l$ of $\u$ in $\p$ and that $U$ is unipotent.
Let $h$ be an element of the centre $\z(\l)$ of $\l$ such that
$\p = \bigoplus_{j \in \Z_{\ge 0}} \p(j;h)$, where
$\p(j;h) = \{x \in \p \mid [h,x] = jx\}$ and $\p(1;h) \ne 0$;
we call such $h$ {\em admissible}.
Note that we have $\u \subseteq \bigoplus_{j \in \Z_{\ge 1}} \p(j;h)$.
Since there are finitely many orbits of $P$ in $\u$, we see that
there is a dense orbit of $C_P(h)$ in $\p(1;h)$ and we let $e$ be
a representative of this orbit; we then say that $e$ is {\em linked}
to $h$.
We define the irreducible $P$-stable subvariety
$$
\CS(h,e) = \overline{P \cdot (h,e)} \subseteq \p \times \u
$$
of $\p \times \u$.
We write $\c_\p(h,e) = \c_\p(e) \cap \c_\p(h)$
for the simultaneous centralizer of $h$ and $e$ in $\p$.

Given a closed subvariety $X$ of an affine space $V$,
we write $\KK(X)$ for the cone of $X$ in $V$, as defined in
\cite[II.~4.2]{kraft}.

The following lemmas are analogues
of results from \cite[\S2]{premet}, the subsequent corollary 
is key in the sequel.

\begin{lem}
\label{lem9}
Let $h$ be admissible and $e$ be linked to $h$.  Then
$$
\KK(\CS(h,e)) \subseteq \CC(\u)
$$
and $\CS(h,e)$ is equidimensional of dimension $\dim P - \dim \c_\p(h,e)$.
In particular, $\KK(\CS(h,e))$ lies in the union of some $\CC(e_i)$ for
which $\dim \CC(e_i) \ge \dim P - \dim \c_\p(h,e)$.
\end{lem}

\begin{proof}
We see that
$$
\CS(h,e) \subseteq \{(x,y) \in \p \times \u \mid [x,y] = y\}.
$$
Therefore, by \cite[II.~4.2 Thm.\ 2]{kraft} and the definition of cones,
$$
\KK(\CS(h,e)) \subseteq \KK(\{(x,y) \in \p \times \u \mid [x,y] = y\}) = \{(x,y) \in \p \times \u \mid [x,y] = 0\}.
$$
We have $\CS(h,e) \subseteq (h + \u) \times \u$ and this
implies that $\KK(\CS(h,e)) \subseteq \u \times \u$.  Hence,
$$
\KK(\CS(h,e)) \subseteq (\u \times \u) \cap \{(x,y) \in \p \times \u \mid [x,y] = 0\} = \CC(\u).
$$
By \cite[II.~4.2 Thm.\ 2]{kraft},
we have that $\KK(\CS(h,e))$ is equidimensional.
The final statement follows easily from the fact
that the irreducible components of
$\CC(\u)$ are of the form $\CC(e_i)$.
\end{proof}

\begin{lem}
\label{lem12}
Let $e \in \u$ and suppose that there
exists admissible $\tilde h \in \z(\l)$
with linked $\tilde e$, such that
$[\tilde h, e] = e$ and $[\c_u(e),\tilde h] = \c_\u(e)$.
Then $(\c_\u(e),e) \subseteq \KK(\CS(\tilde h, \tilde e))$.
\end{lem}

\begin{proof}
Let $H = C_P(\tilde h)$.
The $H$-orbit of $\tilde e$ is dense in $\p(1;\tilde h)$,
and $(\tilde h, H \cdot \tilde e) \subseteq \CS(\tilde h, \tilde e)$, so
we obtain $(\tilde h, \p(1;\tilde h)) \subseteq \CS(\tilde h,\tilde e)$.
Thus $(\tilde h, \bk e) \subseteq \CS(\tilde h, \tilde e)$,
because $e \in \p(1;\tilde h)$.
Consider the $C_U(e)$-orbit
$C_U(e) \cdot \tilde h$ in $\tilde h + \c_\u(e)$.
This is closed in $\tilde h + \c_\u(e)$, because $C_U(e)$ is unipotent.
Since $[\c_u(e),\tilde h] = \c_\u(e)$, we
obtain $C_U(e) \cdot \tilde h =  \tilde h + \c_u(e)$.
Hence, $C_U(e) \cdot (\tilde h, \bk e)
= (\tilde h + \c_\u(e),\bk e) \subseteq \CS(\tilde h,\tilde e)$.
Taking cones we get
$\KK(\tilde h + \c_\u(e),\bk e) \subseteq \KK(\CS(\tilde h,\tilde e))$,
by \cite[II.~4.2 Thm.\ 2]{kraft}.  From the definition of cones we see that
$\KK(\tilde h + \c_\u(e),\bk e) = (\c_\u(e),\bk e)$ and the lemma follows.
\end{proof}

\begin{cor} \label{cor3}
Let $e \in \u$.
Suppose that there exists  admissible $\tilde h \in \z(\l)$
such that $[\tilde h, e] = e$, but $e$ is not linked to $\tilde h$.
Then $\CC(e)$ is not an irreducible component of $\CC(\u)$.
\end{cor}

\begin{proof}
If $[\c_u(e),\tilde h] = \c_\u(e)$, then we can apply Lemma \ref{lem12} to
deduce that $(\c_\u(e),e) \subseteq \KK(\CS(\tilde h, \tilde e))$, where $\tilde e$ is linked
to $\tilde h$.  Then by Lemma \ref{lem9} we have that $\KK(\CS(\tilde h, \tilde e))$ is contained in a union
of $\CC(e_i)$'s of dimension at least $\dim P - \dim \c_\p(\tilde h,\tilde e)$.
Since these $\CC(e_i)$'s are stable under $P$ and $\sigma$
(cf.\ Lemma \ref{lem2}),
we see that $\CC(e)$ is contained in their union.
We note that the conditions $[\tilde h,e] = e$ and
$[\c_u(e),\tilde h] = \c_\u(e)$ imply that
$\dim \c_\p(e) - \dim \c_\u(e) \ge \dim \c_\p(\tilde h,e) >
\dim \c_\p(\tilde h,\tilde e)$, so that
$\dim \CC(e) = \dim P - \dim \c_p(e) + \dim \c_u(e) <
\dim P - \dim \c_\p(\tilde h,\tilde e)$.
Thus $\CC(e)$ is not an irreducible component of $\CC(\u)$.

If $[\c_u(e),\tilde h] \ne \c_\u(e)$, then
$\c_\u(e) \cap \p(0;\tilde h) \ne \{0\}$.
Therefore,
$\c_\u(e) \not\subseteq \overline{P \cdot e} \subseteq
\bigoplus_{j \ge 1} \p(j;\tilde h)$,
so $\CC(e)$ is not an irreducible component of $\CC(\u)$,
by Lemma \ref{lem3}.
\end{proof}

Corollary \ref{cor3}
yields the following strategy to determine the irreducible
components of $\CC(\u)$:

\begin{strategy} \label{strat} $ $
\begin{enumerate}
\item
For each  $i = 1, \ldots,  s$ check
whether $e_i$ is distinguished.
If so, then
$\CC(e_i)$ is an irreducible component, by Lemma \ref{lem8}.
\item Determine all the admissible $h \in \z(\l)$.
For each  $i = 1, \ldots,  s$ check whether
$e_i$ is in $\p(1,h)$ for some admissible
$h$ such  that $e_i$ is not linked to $h$,
so that $\CC(e_i)$ is not an irreducible component of $\CC(\u)$,
by Corollary \ref{cor3}.
\item
For the remaining $i$'s not dealt with in steps (1) and (2),
use ad hoc methods to determine whether
$\CC(e_i)$ is an irreducible component or not.
\end{enumerate}
\end{strategy}

\begin{rem}
\label{R:inf2}
Although we made the assumption that $P$ acts on $\u$
with finitely many orbits above,
the theory goes through with suitable
adaptations when the $P$-orbits can be parameterized
nicely as explained below.

A {\em family of representatives of $P$-orbits in $\u$ over an irreducible variety $X$},
is given by a subset
$e(X) = \{e(t) \mid t \in X\}$ of $\u$ such that: the map
$t \mapsto e(t)$ is an isomorphism from $X$ onto
its image in $\u$; 
and for $t, t' \in X$ distinct, we have
$P \cdot e(t) \ne P \cdot e(t')$ but $\dim P \cdot e(t) = \dim P \cdot e(t')$.

Suppose that the $P$-orbits in $\u$ can be parameterized by a finite number of
families $e_1(X_1),\dots,e_s(X_s)$.
Then all of the theory above has a suitable adaption, when we replace
the single orbits $e_i$ by the families $e_i(X_i)$.  For example, we can define
irreducible varieties $\CC(e_i(X_i))$, and the irreducible components of $\CC(\u)$
are of this form.  For the notion of a family $e(X)$ being linked to an admissible
$h \in \z(\l)$, we require $[h,e(t)]= e(t)$ for all $t \in X$ and $P \cdot e(X)$ to be
dense in $\p(1;h)$, and the subsequent results have similar adaptations.
Therefore, with this assumption on the action of $P$ on $\u$, there is,
a version of Strategy \ref{strat} to determine
the irreducible components of $\CC(\u)$.  We note that
this assumption does hold for the action of a Borel subgroup on the Lie algebra
of its unipotent radical, as explained in \cite[Section 2]{goodwinroehrle}.
\end{rem}

\section{The case of a finite number of $B$-orbits}
\label{s:cu}

This section is devoted to the proof of Theorem \ref{thmmain}.  So in this section $P = B$ is a
Borel subgroup of a simple algebraic group $G$ and $U$ is the unipotent radical of $B$.  Further, we
are assuming that $B$ acts on $\u$ with a finite number of orbits.  As mentioned in the introduction, this means that $G$ is of type
$A_n$ for $n \le 4$ or type $B_2$.  We proceed on a case by case basis using Strategy \ref{strat}
to determine the irreducible components of $\CC(\u)$ and observe that
we obtain the description as given in Theorem \ref{thmmain}.

In each case we give a list of representatives of the $B$-orbits in $\u$.
We calculated these
using an adaptation of the computer program explained
in \cite{goodwinroehrle}; which gives
the same representatives as in \cite[Table 2]{burgsteinhesselink} and as previously
calculated in \cite{kashin}.  The notation used for
these representatives is as follows.  We fix an enumeration $\{\beta_1,\dots,\beta_N\}$ of the roots of $\b$
with respect to a
maximal torus $T$ of $B$,
and for each $\beta_i$ we fix a generator $e_{\beta_i}$
for the corresponding root space.  This enumeration of
the roots is listed, where the roots are given
as vectors with respect to the simple roots as
labelled in \cite[Planches I--IX]{Bo}.
Each of the representatives of the $B$-orbits in $\u$ is of the form
$\sum_{i \in I} e_{\beta_i}$,
where $I \subseteq \{1,\dots,N\}$, and we represent this
element as the coefficient vector with respect to
the $e_{\beta_i}$.

We briefly explain the meaning of an admissible element $h$
in the present setting.
Such $h$ belongs to a maximal toral subalgebra of $\b$
and $\q = \bigoplus_{j \ge 0} \g(j;h)$ is a parabolic subalgebra
of $\g$ such that $\bigoplus_{j > 0} \g(j;h) \subseteq \b \subseteq \q$.
So in this case, Corollary \ref{cor3} says
that if a representative $e$ of a $B$-orbit in $\u$
lies in $\q(1;h) = \g(1;h)$ for such a $\q$
and $e$ is not in the dense $C_B(h)$-orbit in $\q(1;h)$,
then $\CC(e)$ is not an irreducible component of $\CC(\u)$.

\subsection{$G$ is of type $A_1$}

There is just one root of $\b$ and
there are 2 $B$-orbits in $\u$:
the regular and the zero orbit.
Here $\u$ is abelian and $\CC(\u) = \u \times \u$
is irreducible and equal
to $\CC(e)$ where $e$ lies in the regular orbit.

\subsection{$G$ is of type $A_2$}

The roots of $\b$ are given by:

\smallskip

\begin{tabular}{llllll}
$\beta_1$: & 10; & $\beta_2$: & 01; & $\beta_3$: 11. \\
\end{tabular}

\smallskip

\noindent
There are 5 $B$-orbits in $\u$ with representatives:

\smallskip

\begin{tabular}{llllllllll}
$e_1$: & 110; & $e_2$: & 100; & $e_3$: & 010; & $e_4$: & 001; & $e_5$: & 000. \\
\end{tabular}

\smallskip

\noindent
Apart from $e_1$,
each of the $e_i$ lies in $\b(1;h)$ for some admissible $h$,
for which $e_i$ is not linked to $h$.
Therefore, using Strategy \ref{strat},
we get that $\CC(\u) = \CC(e_1)$ is irreducible.

\subsection{$G$ is of type $A_3$}

The roots of $\b$ are given by:

\smallskip

\begin{tabular}{llllllllllll}
$\beta_1$: & 100; & $\beta_2$: & 010; & $\beta_3$: & 001; & $\beta_4$: & 110; & $\beta_5$: & 011; & $\beta_6$: & 111. \\
\end{tabular}

\smallskip

\noindent
There are 16 $B$-orbits  in $\u$ with representatives:

\smallskip

\begin{tabular}{rlrlrlrlrlrl}
$e_1$: & 111000; & $e_2$: & 110000; & $e_3$: & 101010; & $e_4$: & 101000; & $e_5$: & 100010; & $e_6$: & 100000; \\
$e_7$: & 011000; & $e_8$: & 010001; & $e_9$: & 010000; & $e_{10}$: & 001100; & $e_{11}$: & 001000; & $e_{12}$: & 000110; \\
$e_{13}$: & 000100; & $e_{14}$: & 000010; & $e_{15}$: & 000001; & $e_{16}$: & 000000.
\end{tabular}

\smallskip

\noindent
All of the $e_i$ except for  $e_1$, $e_3$, $e_8$  are in
$\b(1;h)$ for some admissible $h$ not linked to $e_i$.
We see that $e_1$ and $e_3$ are distinguished, so $\CC(e_1)$ and $\CC(e_3)$
are irreducible components.  Below we verify by direct calculation that $\CC(e_8)$ is not
an irreducible component.

Consider the pairs of strictly upper triangular matrices
$(x(\alpha,\lambda),y(\alpha,\lambda,a,b,c))$ for
$\alpha,\lambda \in \bk^\times,a,b,c \in \bk$
with entries above the diagonal given by
$$
\left( \begin{array}{ccc} \lambda & 0 & 1 \\ & 1 & 0 \\ & & \alpha\lambda \end{array}, \begin{array}{ccc} \lambda a & b & c \\ & a & \alpha b   \\ & & \alpha\lambda a \end{array} \right).
$$
It is straightforward to check that
$(x(\alpha,\lambda),y(\alpha,\lambda,a,b,c)) \in \CC(\u)$
and that $x(\alpha,\lambda) \in B \cdot e_1$.
Therefore, $(x(\alpha,\lambda),y(\alpha,\lambda,a,b,c)) \in \CC(e_1)$
for all $\alpha,\lambda \in \bk^\times$ and $a,b,c \in \bk$.
Letting $\lambda \to 0$, we see that
$(x(\alpha,0),y(\alpha,0,a,b,c)) \in \CC(e_1)$ for all $\alpha \in \bk^\times$.
We have that $x(\alpha,0) = e_8$ and via a calculation
we see that $\{y(\alpha,0,a,b,c) \mid \alpha \in \bk^\times, a,b,c \in \bk\}$
is a dense subset of
$$
\c_\u(e_8) = \left\{ \begin{array}{ccc} 0 & a & b \\ & c & d \\ & & 0 \end{array} \:\: \vline \:\: a,b,c,d \in \bk \right\}.
$$
Therefore, $(e_8,\c_\u(e_8)) \subseteq \CC(e_1)$ and hence $\CC(e_8) \subseteq \CC(e_1)$.

Putting this all together, we get that $\CC(\u) = \CC(e_1) \cup \CC(e_3)$.

\subsection{$G$ is of type $A_4$}

The roots of $\b$ are given by:

\smallskip

\begin{tabular}{rlrlrlrlrl}
$\beta_1$: & 1000; & $\beta_2$: & 0100; & $\beta_3$: & 0010; & $\beta_4$: & 0001; & $\beta_5$: & 1100;\\
$\beta_6$: & 0110; & $\beta_7$: & 0011; & $\beta_8$: & 1110; & $\beta_9$: & 0111; & $\beta_{10}$: & 1111. \\
\end{tabular}

\smallskip

There are 61 $B$-orbits  in $\u$ with representatives:

\smallskip

\begin{tabular}{rlrlrlrlrlrl}
$e_1$: & 1111000000; & $e_2$: & 1110000000; & $e_3$: & 1101000000; & $e_4$: & 1100001000; \\
$e_5$: & 1100001000; & $e_6$: & 1100000000; & $e_7$: & 1011010000; & $e_8$: & 1011000000; \\
$e_9$: & 1010010010; & $e_{10}$: & 1010010000; & $e_{11}$: & 1010000010; & $e_{12}$: & 1010000000; \\
$e_{13}$: & 1001010000; & $e_{14}$: & 1001000010; & $e_{15}$: & 1001000000; & $e_{16}$: & 1000011000; \\
$e_{17}$: & 1000010000; & $e_{18}$: & 1000001010; & $e_{19}$: & 1000001000; & $e_{20}$: & 1000000010; \\
$e_{21}$: & 1000000000; & $e_{22}$: & 0111000000; & $e_{23}$: & 0110000001; & $e_{24}$: & 0110000000; \\
$e_{25}$: & 0101000110; & $e_{26}$: & 0101000100; & $e_{27}$: & 0101000010; & $e_{28}$: & 0101000000; \\
$e_{29}$: & 0100001100; & $e_{30}$: & 0100001000; & $e_{31}$: & 0100000100; & $e_{32}$: & 0100000001; \\
$e_{33}$: & 0100000000; & $e_{34}$: & 0011100000; & $e_{35}$: & 0011000000; & $e_{36}$: & 0010100010; \\
$e_{37}$: & 0010100000; & $e_{38}$: & 0010000010; & $e_{39}$: & 0010000001;  & $e_{40}$: & 0010000000; \\
$e_{41}$: & 0001110000; & $e_{42}$: & 0001100100; & $e_{43}$: & 0001100000; & $e_{44}$: & 0001010000; \\
$e_{45}$: & 0001000100; & $e_{46}$: & 0001000000; & $e_{47}$: & 0000111000; & $e_{48}$: & 0000110000; \\
$e_{49}$: & 0000101000; & $e_{50}$: & 0000100010; & $e_{51}$: & 0000100000; & $e_{52}$: & 0000011000; \\
$e_{53}$: & 0000010001; & $e_{54}$: & 0000010000; & $e_{55}$: & 0000001100; & $e_{56}$: & 0000001000; \\
$e_{57}$: & 0000000110; & $e_{58}$: & 0000000100; & $e_{59}$: & 0000000010; & $e_{60}$: & 0000000001; \\
$e_{61}$: & 0000000000.
\end{tabular}

\smallskip

\noindent
Except for $e_1$, $e_3$, $e_7$, $e_9$, $e_{14}$, $e_{23}$ and $e_{25}$,
we can check that each $e_i$ lies in $\b(1;h)$ for some admissible $h$ not linked to $e_i$.
The representatives $e_1$, $e_3$, $e_7$, $e_9$ and $e_{25}$ are distinguished,
so the corresponding $\CC(e_i)$'s are irreducible components of $\CC(\u)$.
Below we verify by direct calculation that $\CC(e_8)$ and $\CC(e_{14})$ are not
irreducible components.

We have
$$
\left( \begin{array}{cccc} 1 & 0 & 0 & 0  \\ & \alpha\lambda & \lambda & 1 \\ & & 0 & 0 \\ & & & 1 \end{array}, \begin{array}{cccc} a & c & e & f \\ & \alpha\lambda a & \lambda a & a+e  \\ & & 0 & \alpha^{-1}(a-b) \\ & & & b \end{array} \right) \in \CC(e_3)
$$
for all $\alpha, \lambda \in \bk^\times$ and $a,b,c,e,f \in \bk$,
and
$$
\c_\u(e_{14}) \: = \left\{\begin{array}{cccc} a & c & e & f \\ & 0 & 0 & a+e  \\ & & 0 & d \\ & & & b \end{array} \:\: \vline \:\:  a,b,c,d,e,f \in \bk \right\}.
$$
Letting $\lambda \to 0$, we see that $\CC(e_{14}) \subseteq \CC(e_3)$.

Similarly,
we have
$$
\left( \begin{array}{cccc} \lambda & 0 & 0 & 1  \\ & 1 & 0 & 0 \\ & & 1 & 0 \\ & & & \alpha\lambda \end{array}, \begin{array}{cccc} \lambda a & \lambda b & c & e \\ & a & b & \alpha c  \\ & & a & \lambda \alpha b \\ & & & \lambda \alpha a \end{array} \right) \in \CC(e_1)
$$
for all $\alpha, \lambda \in \bk^\times$ and $a,b,c,e \in \bk$,
and
$$
\c_\u(e_{23}) \:= \left\{\begin{array}{cccc} 0 & 0 & c & e \\ & a & b & d  \\ & & a & 0 \\ & & & 0 \end{array} \:\: \vline \:\:  a,b,c,d,e \in \bk\right\}.
$$
Letting $\lambda \to 0$, we see that $\CC(e_{23}) \subseteq \CC(e_1)$.

Combining the above,
the decomposition of $\CC(\u)$ into irreducible
components is given by
$
\CC(\u) = \CC(e_1) \cup \CC(e_3) \cup \CC(e_7) \cup \CC(e_9) \cup \CC(e_{25}).
$

\subsection{$G$ is of type $B_2$}

The roots of $\b$ are given by:

\smallskip

\begin{tabular}{llllllll}
$\beta_1$: & 10; & $\beta_2$: & 01; & $\beta_3$: & 11; & $\beta_4$: 12. \\
\end{tabular}

\smallskip

\noindent
There are 7 $B$-orbits  in $\u$ with representatives:

\smallskip

\begin{tabular}{llllllllllllllll}
$e_1$: & 1100; & $e_2$: & 1001; & $e_3$: & 1000; & $e_4$: & 0100; &
$e_5$: & 0010; & $e_6$: & 0001; & $e_7$: & 0000. \\
\end{tabular}

\smallskip

\noindent
The two orbit representatives $e_1$ and $e_2$ are distinguished.
Each of the other orbit representatives $e_i$ lies in
$\b(1;h)$ for some $h$ which is not linked to $e_i$.
So, using Strategy \ref{strat}, we have
$\CC(\u) = \CC(e_1) \cup \CC(e_2)$.

\begin{rem}
\label{rem:parabolic}
All of the material in Section \ref{s:prelims} is valid when $P$ is a parabolic
subgroup of a reductive algebraic group $G$ and $U$ is the unipotent radical of
$P$.  We note, however,
that in contrast to Theorem \ref{thmmain},
$\CC(\u)$ is not equidimensional in general
when there are finitely many $P$-orbits in $\u$.
In fact,
the difference in the dimensions of irreducible components can
be arbitrarily large, as shown in the example below.

Let $m \ge 2$ be an integer and let
$P$ be the
parabolic subgroup of $\GL_{m+2}(\bk)$, which is the stabilizer of a flag
of subspaces $\bk \subseteq \bk^2 \subseteq \bk^{m+2}$
in $\bk^{m+2}$.  Then $P$ admits only a finite number of
orbits on the Lie algebra of its unipotent
radical $\u$, see \cite{hilleroehrle}.
However, one can calculate that $\CC(\u)$ has two irreducible components
of dimensions $4m+1$ and $3m+2$.
\end{rem}

\section{The case of an infinite number of $B$-orbits}
\label{s:cu2}

We continue using the notation from the last section, but
we remove the assumption that $B$ acts on $\u$ with a finite number of orbits.
Also, we use the notation for families of $B$-orbits $e(X)$ in $\u$,
as explained in Remark \ref{R:inf2}.

We begin by observing that
the analogue of Theorem \ref{thmmain} does not hold
when there are infinitely many $B$-orbits in $\u$.

\begin{lem}
\label{L:notequi}
Suppose that $B$ acts on $\u$ with an infinite number of orbits.
Then $\CC(\u)$ is not equidimensional.
\end{lem}

\begin{proof}
Let $e \in \u$ be in the regular nilpotent orbit.
Then by Lemma \ref{lem1}(ii), we have that
$\CC(e)$ is an irreducible component of $\CC(\u)$ of dimension $\dim B$.

Since $B$ acts on $\u$ with an infinite number
of orbits, there is a family of $B$-orbits
$e(X)$ of $B$-orbits parameterized
by some irreducible variety $X$ of positive
dimension such that $\c_\b(e(t)) = \c_\u(e(t))$
for all $t \in X$; this assertion
follows from the fact that it can be explicitly checked for the minimal
infinite cases considered below.  Then we have that
$\dim \CC(e(X)) = \dim B + \dim X$.
Thus, there must be an irreducible component of
$\CC(\u)$ of dimension strictly larger than $\dim B$.
\end{proof}

We move on to describe the irreducible components
of $\CC(\u)$ for the cases where $\g$ is of type
$A_5$, $B_3$, $C_3$, $D_4$ and $G_2$.
These are the minimal cases in which there is an
infinite number of $B$-orbits in $\u$.
We have determined the irreducible
components using the
adaptation of Strategy \ref{strat}, as
discussed in Remark \ref{R:inf2}.
The calculations are very similar in spirit to those discussed in Section \ref{s:cu}, so
we omit the details.  We use a parameterization of orbits given by
the programme
from \cite{goodwinroehrle};
most of this information can also be extracted from
\cite{burgsteinhesselink}.

From the descriptions given below, we see that the structure of $\CC(\u)$ is
already rather complicated, and there does not appear to be a nice way to
parameterize the irreducible components already in these minimal infinite
cases.  We have investigated the possibility
of doing this in terms of a suitable
notion of {\em distinguished families} of $B$-orbits in $\u$.
However, the natural candidates do not
give the irreducible components, as desired.

\subsection{$G$ is of type $A_5$}

The roots of $\b$ are given by:

\smallskip

\begin{tabular}{rlrlrlrlrl}
$\beta_1$: & 10000; & $\beta_2$: & 01000; & $\beta_3$: & 00100; & $\beta_4$: & 00010; & $\beta_5$: & 00001;\\
$\beta_6$: & 11000; & $\beta_7$: & 01100; & $\beta_8$: & 00110; & $\beta_9$: & 00011; & $\beta_{10}$: & 11100; \\
$\beta_{11}$: & 01110; & $\beta_{12}$: & 00111; & $\beta_{13}$: & 11110; & $\beta_{14}$: & 01111; & $\beta_{15}$: & 11111. \\
\end{tabular}

\smallskip

\noindent
The $B$-orbits in $\u$ are given by one $1$-dimensional family
$e_{29}(\bk^\times)$ given by $t \mapsto 1010101010t0000$
and 274 other orbits.
We have that $\CC(e_{29}(\bk^\times))$ is an irreducible component of dimension 21 and there are 12
irreducible components of dimension 20 given by $\CC(e_i)$,
where $e_i$ is one of the following:

\smallskip

\begin{tabular}{rlrlrlrlrl}
$e_1$: & 111110000000000; & $e_3$:    & 111010001000000; & $e_7$:    & 110110010000000; \\
$e_8$:    & 110110000001000; & $e_{10}$: & 110100010001000; & $e_{23}$: & 101110100000000; \\
$e_{25}$:  & 101100100000010; & $e_{53}$:  & 100100000011010; & $e_{94}$:  & 011010001000100; \\
$e_{103}$: & 010110010100000; & $e_{107}$: & 010100010101000; & $e_{119}$: & 010010000101100. \\
\end{tabular}

\subsection{$G$ is of type $B_3$}

The roots of $\b$ are given by:

\smallskip

\begin{tabular}{llllllllll}
$\beta_1$: & 100; & $\beta_2$: & 010; & $\beta_3$: & 001; & $\beta_4$: & 110; & $\beta_5$: & 011;\\
$\beta_6$: & 111; & $\beta_7$: & 012; & $\beta_8$: & 112; & $\beta_9$: & 122. \\
\end{tabular}

\smallskip

\noindent
The $B$-orbits in $\u$ are given by one $1$-dimensional family
$e_{12}(\bk^\times)$ given by $t \mapsto 0100011t0$
and $34$ other orbits.
We have that $\CC(e_{12}(\bk^\times))$
is an irreducible component of dimension 13 and there are 4
irreducible components of dimension 12 given by $\CC(e_i)$,
where $e_i$ is one of the following:

\smallskip

\begin{tabular}{llllllllllll}
$e_1$: & 111000000; & $e_2$:    & 110000100; & $e_4$: & 101010000; &
$e_5$: & 101000001.\\
\end{tabular}

\subsection{$G$ is of type $C_3$}

The roots of $\b$ are given by:

\smallskip

\begin{tabular}{llllllllll}
$\beta_1$: & 100; & $\beta_2$: & 010; & $\beta_3$: & 001; & $\beta_4$: & 110; & $\beta_5$: & 011;\\
$\beta_6$: & 111; & $\beta_7$: & 021; & $\beta_8$: & 121; & $\beta_9$: & 221. \\
\end{tabular}

\smallskip

\noindent
The $B$-orbits in $\u$ are given by one $1$-dimensional family
$e_4(\bk^\times)$ given by $t \mapsto 101010t00$
and 34 other orbits.
We have that $\CC(e_4(\bk^\times))$
is an irreducible component of dimension 13 and there are 3
irreducible components of dimension 12 given by $\CC(e_i)$ where $e_i$ is one of the following:

\smallskip

\begin{tabular}{llllllllllll}
$e_1$: & 111000000; & $e_2$:  & 110000100; & $e_{12}$: & 011000001. \\
\end{tabular}

\subsection{$G$ is of type $D_4$}

The roots of $\b$ are given by:

\smallskip

\begin{tabular}{rlrlrlrlrlrl}
$\beta_1$: & 1000; & $\beta_2$: & 0100; & $\beta_3$: & 0010; & $\beta_4$: & 0001; & $\beta_5$: & 1100; & $\beta_6$: & 0110; \\
$\beta_7$: & 0101; & $\beta_8$: & 1110; & $\beta_9$: & 1101; & $\beta_{10}$: & 0111; & $\beta_{11}$: & 1111; & $\beta_{12}$: & 1211. \\
\end{tabular}

\smallskip

\noindent
The $B$-orbits in $\u$ are given by two $1$-dimensional families
$e_8(\bk^\times)$ given by $t \mapsto  101101t00000$ and
$e_{37}(\bk^\times)$ given by $t \mapsto  0100000111t0$
and 98 other orbits.
We have that $\CC(e_8(\bk^\times))$ and $\CC(e_{37}(\bk^\times))$
are irreducible components of dimension 17 and there are 4
irreducible components of dimension 16 given by $\CC(e_i)$,
where $e_i$ is one of the following:

\smallskip

\begin{tabular}{llllllllllll}
$e_1$: & 111100000000; & $e_2$:  & 111000000100; & $e_4$: & 110100000100; & $e_{31}$: & 011100001000.   \\
\end{tabular}

\subsection{$G$ is of type $G_2$}

The roots of $\b$ are given by:

\smallskip

\begin{tabular}{llllllllllll}
$\beta_1$: & 10; & $\beta_2$: & 01; & $\beta_3$: & 11; & $\beta_4$: & 21; & $\beta_5$: & 31; & $\beta_6$: & 32.\\
\end{tabular}

\smallskip

\noindent
The $B$-orbits in $\u$ are given by one $1$-dimensional family
$e_4(\bk^\times)$ given by $t \mapsto  0101t0$
and 11 other orbits.
We have that $\CC(e_4(\bk^\times))$
is an irreducible component of dimension 9 and there are 2
irreducible components of dimension 8 given by $\CC(e_i)$,
where $e_i$ is one of the following:

\smallskip

\begin{tabular}{llllllllllll}
$e_1$: & 110000; & $e_2$:  & 100001. \\
\end{tabular}


\bigskip

\noindent
{\bf Acknowledgments}:
Part of the research for this paper was carried out while
both authors were staying at the Mathematical Research Institute
Oberwolfach supported by the ``Research in Pairs'' programme.


\end{document}